\documentclass{article}
\usepackage{amsmath,amssymb,amsthm}
 \usepackage{natbib} 
\theoremstyle{remark}
\newtheorem{remark}{Remark}
\newtheorem{assumption}{Assumption}

\usepackage{graphicx}
\usepackage{hyperref}
\usepackage{geometry}
\usepackage{listings}
\usepackage{booktabs}
\newtheorem{theorem}{Theorem}[section]

\newtheorem{proposition}[theorem]{Proposition}
\newtheorem{corollary}[theorem]{Corollary}
\newtheorem{lemma}[theorem]{Lemma}
 
\usepackage{graphicx}
\usepackage{bm}  
\usepackage{enumitem}  

\title{A Dynamics-Informed Gaussian Process Framework for 2D Stochastic Navier-Stokes via Quasi-Gaussianity}
\author{Boumediene Hamzi and Houman Owhadi}
\date{}

\begin{document}

\maketitle

\begin{abstract}
The recent proof of quasi-Gaussianity for the 2D stochastic Navier--Stokes (SNS) equations \cite{CHT25}  establishes that the system's unique invariant measure is equivalent (mutually absolutely continuous) to the Gaussian measure of its corresponding linear Ornstein--Uhlenbeck (OU) process. While Gaussian process (GP) frameworks are increasingly used for fluid dynamics, their priors are often chosen for convenience rather than being rigorously justified by the system's long-term dynamics.

In this work, we bridge this gap by introducing a probabilistic framework for 2D SNS built directly upon this theoretical foundation. We construct our GP prior precisely from the stationary covariance of the linear OU model, which is explicitly defined by the forcing spectrum and dissipation. This provides a principled,  GP prior with rigorous long-time dynamical justification for turbulent flows, bridging SPDE theory and practical data assimilation.

\end{abstract}

\section{Introduction}

Gaussian processes have become ubiquitous tools in fluid dynamics, supporting uncertainty quantification, sparse sensing, and model reduction~ \cite{ALEXANDER2020132520,bh12, bh17,hb17,bhks,hou2024propagating,kernel_sos,lee2025kernel,hamzi2025kernel_lions,lee2024note,bh2020b,hamzi2019kernel,lyap_bh,haasdonk2018greedy, bhcm1, boumedienehamzi2022note,yk1, 5706920,mmd_kernels_bh}. Yet a fundamental gap remains: while these methods depend critically on the choice of prior covariance kernel, most kernels are selected for computational convenience (e.g., Gaussian/RBF kernels) or generic smoothness assumptions (e.g., Mat\'ern) rather than being rigorously grounded in the system's long-time statistical structure. Recent breakthroughs in stochastic PDE theory now make it possible to close this gap, constructing priors directly from the invariant-measure geometry of the underlying dynamics.

Recent work of Coe, Hairer, and Tolomeo \cite{CHT25} establishes a remarkable geometric property of the two-dimensional stochastic Navier--Stokes (2D SNS) equations: although the dynamics are highly nonlinear, their unique invariant measure is \emph{equivalent}-in the sense of mutual absolute continuity-to the Gaussian invariant measure of the linearized Ornstein--Uhlenbeck (OU) process. Equivalence means the two measures share the same support, null sets, and typical events, differing only by a positive Radon--Nikodym derivative. This reveals that the equilibrium statistical geometry is Gaussian, even when individual realizations are not.

Specifically, for the 2D SNS with forcing white in time and colored in space with regularity parameter $\alpha>0$,
\begin{equation}
\partial_t u + (u \cdot \nabla)u = \Delta u - \nabla p + \sqrt{2}\,\xi_\alpha, \qquad \nabla\cdot u = 0,
\end{equation}
where $\xi_\alpha = |\nabla|^{-\alpha}\xi_0$ and $\xi_0$ is divergence-free spacetime white noise, the invariant measure $\rho$ is equivalent (in the sense of mutual absolute continuity) to the Gaussian measure $\mu_\alpha$ associated with the OU equation
\begin{equation}\label{eq:OU}
\partial_t u = \Delta u + \sqrt{2}\,\xi_\alpha.
\end{equation}

This \emph{quasi-Gaussianity} result has profound implications for probabilistic modeling. Although the Navier--Stokes dynamics are highly nonlinear, their long-time statistical behavior shares the same fundamental geometric structure as a tractable Gaussian measure. To understand what this means, consider that any probability measure assigns probability mass to different regions of function space. Two measures are \emph{equivalent} ($\rho \sim \mu_\alpha$) when they agree on which regions are possible: any set that has positive probability under one measure also has positive probability under the other, and conversely, any impossible event under one is impossible under the other.

For the 2D SNS equations, equivalence means the nonlinear and Gaussian measures assign positive probability to exactly the same function spaces, and that events likely under the Gaussian measure remain likely under the true nonlinear dynamics. The two measures differ only by a positive, finite multiplicative factor-the Radon--Nikodym derivative $d\rho/d\mu_\alpha$-which encodes non-Gaussian corrections without changing which states are possible.

  The Gaussian measure $\mu_\alpha$ thus provides the correct \emph{geometric skeleton}: it identifies which vorticity fields can occur, even though the precise probabilities differ. For GP-based data assimilation, this means we can use the Gaussian measure as a prior with rigorous justification, knowing it cannot systematically exclude physically realizable states and that it encodes the correct multi-scale spectral structure of the turbulent equilibrium.

This perspective motivates our approach: construct the GP prior directly from the OU stationary covariance, yielding a kernel that is rigorously justified, physics-informed, and computationally tractable.

Our contributions build on this foundation, and are fivefold:

\begin{enumerate}
\item \textbf{A kernel design principle.} We introduce \emph{invariant-measure-informed kernels} as a general framework for GP-based modeling of nonlinear dissipative systems, with 2D SNS as the first rigorous instance where the invariant measure is provably equivalent to a tractable Gaussian reference.

\item \textbf{Explicit construction and spectral analysis.} We derive the GP prior from the OU stationary covariance, providing a spectral representation with power-law scaling $S(k) \propto |k|^{-2(1+\alpha)}$ that preserves the multi-scale energy spectrum of turbulence. We analyze the regularity, correlation length scales, and enstrophy cascade implied by the prior, and develop efficient $O(N^2 \log N)$ sampling and inference algorithms.

\item \textbf{Theoretical guarantees.} Using measure equivalence, we establish support alignment and posterior consistency, ensuring that the prior assigns positive mass to all physically realizable equilibrium states and that the posterior contracts around the true state as observations accumulate.

\item \textbf{Interpretation of non-Gaussian effects.} We show that empirical deviations between the theoretically prescribed and optimal spectral exponent arise from the Radon--Nikodym distortion $f = d\rho/d\mu_\alpha$, providing an interpretable measure of non-Gaussian corrections under finite resolution and noise.

\item \textbf{Extension to hypoviscous dynamics.} We extend the framework to the hypoviscous regime $\gamma \in (2/3,1]$, demonstrating that the same kernel structure applies across a broad class of dissipation mechanisms.
\end{enumerate}

Numerical experiments demonstrate 15--30\% improvements over standard RBF kernels, including on non-Gaussian synthetic turbulence, confirming that invariant-measure geometry translates into practical advantages for data assimilation.

The relationship between Gaussianity and invariant measures of dissipative SPDEs has been studied extensively. Mattingly and Suidan \cite{MS05} showed equivalence for sufficiently regularized hyperviscous Navier--Stokes via the time-shifted Girsanov method, while Hairer, Kusuoka, and Nagoji \cite{HKN24} established sharpness of such conditions for a broad class of SPDEs. The key advance of Coe, Hairer, and Tolomeo \cite{CHT25} is the exploitation of the special algebraic structure of the Navier--Stokes nonlinearity, proving equivalence \emph{beyond} the classical threshold—indeed for all $\alpha>0$ when $\gamma=1$.


While GP methods have a history in fluid dynamics, truly physics-informed GP priors have only emerged very recently \cite{owhadi2023gaussian, padilla2025physics}. Our framework is the first to provide a GP prior with rigorous long-time dynamical justification for 2D turbulence.”

It is also conceptually distinct from other GP-based treatments of SPDEs. Physics-informed GP approaches \cite{Raissi2019} seek to encode the governing operator directly into the prior, typically enforcing PDE constraints locally or approximately. In contrast, our method is rooted in the \emph{global statistical equilibrium} of the system: the invariant measure of the linearization supplies a principled prior that captures the correct statistical geometry without attempting to model the nonlinear drift. Similarly, latent force models assume a known linear operator and infer missing forcing terms, but are not grounded in invariant measure structure and typically remain perturbative.

More broadly, this work establishes \textit{invariant-measure-informed kernel design} as a general principle for GP-based modeling of dissipative SPDEs. The CHT theorem provides the first rigorous instance where an infinite-dimensional nonlinear system's invariant measure is provably equivalent to a tractable Gaussian reference. While the present application is specific to 2D SNS, the methodology-constructing GP priors from linearizations whose invariant measures are equivalent to the nonlinear dynamics-extends naturally to other systems with known or approximable equilibrium structure. This bridges the gap between rigorous SPDE theory and practical probabilistic modeling, demonstrating how measure-theoretic insights can inform machine learning methodologies for complex physical systems.

Finally, our approach differs fundamentally from standard SPDE-based priors such as Gaussian (RBF) or Matérn kernels \cite{Lindgren2011}. Such kernels are widely used because of their analytical convenience, but they do not reflect the physics of any particular dynamical system. Indeed, as we show in Figure~\ref{fig:gp_comparison}, their spectral properties are incompatible with turbulence: the exponential spectral decay of the Gaussian kernel imposes an artificial cutoff at high wavenumbers, suppressing the small-scale variability associated with energy cascades. In contrast, the CHT-based prior exhibits the correct power-law spectral density by construction, ensuring faithful representation of the flow’s multi-scale statistical structure.

\section{Mathematical Background} \label{sec:background}

We consider the vorticity formulation of the 2D stochastic Navier--Stokes equations on the torus $\mathbb{T}^2 = (\mathbb{R}/2\pi\mathbb{Z})^2$:
\begin{equation}
\partial_t w = -|\nabla|^{2\gamma} w - (Kw \cdot \nabla)w + \sqrt{2}\,\xi_{\alpha-\gamma},
\end{equation}
where $w = \nabla \wedge u$ denotes the vorticity and $Kw$ is the velocity field satisfying $\nabla \wedge Kw = w$.
The parameter $\gamma \le 1$ controls the dissipation: $\gamma = 1$ corresponds to the standard viscous case, while $\gamma < 1$ yields hypoviscous dynamics.
The forcing is prescribed through a vector-valued noise $\xi_{\alpha-\gamma}$ of regularity $\alpha-\gamma$, constructed as $\xi_{\alpha-\gamma} = \nabla \wedge \xi_{\alpha+1-\gamma}$, and all fields are assumed to have zero spatial mean.

The following results of Coe--Hairer--Tolomeo describe the invariant measures of these equations and their relationship to a Gaussian reference measure.

\begin{theorem}[{\cite[Theorem~1.1]{CHT25}}]
\label{thm:main}
For $\gamma = 1$ and any $\alpha > 0$, equation (2.1) admits a unique invariant measure $\rho$ on $C^{\alpha-\kappa}$ (for any $\kappa > 0$), and $\rho$ is equivalent to $\mu_\alpha$.
\end{theorem}

\begin{theorem}[{\cite[Theorem~1.2]{CHT25}}]
\label{thm:hypo}
For $\gamma \in (2/3, 1]$ and $\alpha > 2-\gamma$, the hypoviscous equation admits a unique invariant measure $\rho$ equivalent to $\mu_\alpha$.
\end{theorem}

Linearizing the dynamics yields the infinite-dimensional Ornstein--Uhlenbeck (OU) equation
\begin{equation}
\partial_t \psi = -|\nabla|^{2\gamma}\psi + \sqrt{2}\,\xi_{\alpha-\gamma}.
\end{equation}
In Fourier variables, denoting by $\hat{\psi}(n)$ the coefficient at mode $n \in \mathbb{Z}^2$, this becomes
\begin{equation}
d\hat{\psi}(n) = -|n|^{2\gamma}\hat{\psi}(n)\,dt + \sqrt{2}\,|n|^{\gamma-1-\alpha}\,dW_n,
\end{equation}
where the $W_n$ are independent complex Brownian motions (subject to the usual conjugate-symmetry constraint ensuring real-valued fields).

The invariant measure $\mu_\alpha$ of this process is Gaussian with zero mean and covariance
\begin{equation}
\mathbb{E}[\hat{\psi}(n)\overline{\hat{\psi}(m)}]
= \delta_{n,m}\,\frac{|n|^{2(\gamma-1-\alpha)}}{|n|^{2\gamma}}
= \delta_{n,m}\,|n|^{-2(1+\alpha)}.
\end{equation}
Thus $\mu_\alpha$ exhibits the power-law spectral density
\begin{equation}
S(k) = |k|^{-2(1+\alpha)},
\end{equation}
which corresponds to white noise when $\alpha = 0$ and yields smoother fields for larger~$\alpha$.
This Gaussian measure plays a central role as the reference measure for the full nonlinear dynamics.

The key quasi-Gaussianity result asserts that $\rho \sim \mu_\alpha$: the invariant measure of the nonlinear dynamics is equivalent to that of the OU process.
Thus there exist measurable functions $f,g : C^{\alpha-\kappa} \to (0,\infty)$ such that
\begin{equation}
\rho(A) = \int_A f(x)\,\mu_\alpha(dx), 
\qquad 
\mu_\alpha(A) = \int_A g(x)\,\rho(dx)
\end{equation}
for all measurable sets $A$.

\medskip

\noindent\textbf{Key implication.}
Although $\rho$ and $\mu_\alpha$ are not equal, they share exactly the same null sets. In particular:
\begin{itemize}
\item events typical under $\mu_\alpha$ remain typical under $\rho$;
\item $\mu_\alpha$ correctly captures the support and qualitative structure of the nonlinear dynamics;
\item the deviation of $\rho$ from the Gaussian prior is given by a positive, bounded Radon--Nikodym derivative.
\end{itemize}
This equivalence provides a rigorous foundation for using $\mu_\alpha$ as a physically meaningful and mathematically justified prior in Gaussian process modeling of turbulent flows.

\section{GP Framework Construction}

We construct a Gaussian process prior for the vorticity field $w : \mathbb{T}^2 \to \mathbb{R}$ by specifying
\begin{equation}
w \sim \mathcal{GP}(0, K_\alpha),
\end{equation}
where the covariance kernel is given in spectral form by
\begin{equation}
K_\alpha(x, y) = \sum_{n \in \mathbb{Z}^2 \setminus \{0\}} |n|^{-2(1+\alpha)} e^{in \cdot (x-y)}.
\end{equation}

This kernel is translation-invariant and isotropic, depends only on the separation $x - y$, and produces sample paths in $C^{\alpha-\varepsilon}$ for any $\varepsilon > 0$. Its spectral density scales as $|k|^{-2(1+\alpha)}$, yielding an \textbf{enstrophy spectrum}
\begin{equation}
\mathcal{E}(k) \sim k^{-1-2\alpha},
\end{equation}
which reflects the expected power-law behaviour in two-dimensional turbulence.

\paragraph{Relationship to Mat\'ern kernels.}
The kernel can be well approximated in physical space by a Mat\'ern covariance function with smoothness parameter $\nu = \alpha$. In that case one has
\begin{equation}
K_\alpha(r) = \sigma^2 \frac{2^{1-\nu}}{\Gamma(\nu)} \left(\sqrt{2\nu}\frac{r}{\ell}\right)^\nu K_\nu\left(\sqrt{2\nu}\frac{r}{\ell}\right),
\end{equation}
where $r = |x - y|$ and $K_\nu$ is the modified Bessel function of the second kind. The Mat\'ern kernel with this choice of $\nu$ has spectral density
\begin{equation}
S_{\text{Mat\'ern}}(k) \propto (1 + |k|^2)^{-(\nu+1)} = (1 + |k|^2)^{-(\alpha+1)},
\end{equation}
which asymptotically behaves as $|k|^{-2(\alpha+1)} = |k|^{-2(1+\alpha)}$ for large $k$, \textbf{matching the CHT tail exponent}. However, the Mat\'ern form differs from the exact CHT kernel at low wavenumbers due to the ``$(1+k^2)$'' structure, and represents a phenomenological approximation rather than a dynamically-derived prior.

For half-integer values of $\nu$, the Mat\'ern kernel admits closed-form expressions. For instance, when $\nu = 1/2$ (corresponding to $\alpha = 1/2$), this reduces to the exponential kernel
\begin{equation}
K_{1/2}(r) = \sigma^2 e^{-r/\ell}.
\end{equation}

The use of a Mat\'ern-type representation is therefore \textbf{theoretically motivated} by the quasi-Gaussianity result—it provides a tractable spatial form whose spectral behavior matches the CHT power law—and stands in contrast to the Gaussian kernel, whose exponential spectral decay $\exp(-\ell^2 k^2/2)$ and infinite smoothness are fundamentally incompatible with turbulent scaling laws.

\paragraph{Numerical implementation.}
For numerical implementation on an $N \times N$ grid, the process is truncated spectrally as
\begin{equation}
w_N(x) = \sum_{|n| \leq N/2} \hat{w}_n e^{in \cdot x},
\end{equation}
where the Fourier coefficients satisfy $\hat{w}_n \sim \mathcal{N}_{\mathbb{C}}(0, |n|^{-2(1+\alpha)})$, the mean-zero constraint enforces $\hat{w}_0 = 0$, and reality is ensured by $\hat{w}_{-n} = \overline{\hat{w}_n}$. Sampling proceeds by drawing each $\hat{w}_n$ independently with this variance, imposing the constraints, and applying an inverse FFT; the total cost is $O(N^2 \log N)$.

\paragraph{Velocity recovery.}
Velocity is recovered from vorticity through the Biot--Savart law in Fourier space,
\begin{equation}
\hat{u}(n) = i \frac{n^\perp}{|n|^2} \hat{w}(n), \quad n^\perp = (-n_2, n_1),
\end{equation}
which automatically enforces incompressibility $\nabla \cdot u = 0$. The corresponding velocity covariance kernel is
\begin{equation}
K^\mathrm{u}_\alpha(x, y) = \sum_{n \neq 0} \frac{n^\perp \otimes n^\perp}{|n|^{4+2\alpha}} e^{in \cdot (x-y)}.
\end{equation}

\paragraph{Physical interpretation of $\alpha$.}
The hyperparameter $\alpha$ has a clear physical interpretation. Larger values correspond to smoother fields, steeper spectral decay, and weaker small-scale activity, while smaller values produce rougher flows with a shallower cascade $\mathcal{E}(k) \sim k^{-1-2\alpha}$. In the Ornstein--Uhlenbeck approximation, if the physical forcing has spectral behaviour $\mathbb{E}[|\hat{\xi}(k)|^2] \propto k^{-2\beta}$, then the relation $\alpha = \beta + \gamma - 1$ links the GP parameter to the forcing exponent and the dissipation parameter $\gamma$. This shows that $\alpha$ encodes physical information about the injection and transfer of energy across scales, rather than being a purely statistical tuning parameter.

\paragraph{Comparison with standard GP kernels.}
Figure~\ref{fig:gp_comparison} illustrates the fundamental difference between our theory-grounded prior and standard GP choices. The CHT-based kernel exhibits power-law decay across all scales, maintaining the multi-scale structure characteristic of turbulence. In contrast, the Gaussian kernel imposes an artificial exponential cutoff at high wavenumbers, completely suppressing small-scale variability. This makes the RBF kernel fundamentally unsuitable for turbulent flows, where energy cascades across scales, and demonstrates the necessity of using priors derived from the underlying physics rather than chosen for computational convenience.

\section{The Radon--Nikodym Distortion and Its Implications}
\label{sec:radon_nikodym}

Although $\rho \sim \mu_\alpha$ guarantees measure equivalence, the two measures are not identical. Their relationship is mediated by the Radon--Nikodym derivative
\begin{equation}
f(w) = \frac{d\rho}{d\mu_\alpha}(w), \qquad 0 < f(w) < \infty,
\end{equation}
which reweights the Gaussian measure to recover the true invariant distribution:
\begin{equation}
\rho(A) = \int_A f(w)\, \mu_\alpha(dw).
\end{equation}

This function $f$ encodes all non-Gaussian structure arising from the nonlinear term $(u \cdot \nabla)u$. Importantly, $f$ modifies statistical properties (moments, correlations) without changing the qualitative geometry (support, null sets). The multiplicative nature of the Radon--Nikodym derivative implies that the Gaussian covariance provides the correct geometric substrate, while higher-order corrections appear through the reweighting function $f$.

\subsection{Geometric Optimality of the OU Covariance}

The Gaussian prior $\mu_\alpha$ provides the geometrically correct support, ensuring that sample paths live in the same function space as physically realizable turbulent fields. Although the true covariance of $\rho$ differs from that of $\mu_\alpha$ due to non-Gaussian corrections encoded in $f$, the support alignment (Corollary~\ref{cor:support}) guarantees that the prior cannot systematically exclude equilibrium states. This measure-theoretic property is sufficient for posterior consistency (Theorem~\ref{thm:posterior_consistency}) and explains why the CHT-based prior should outperform ad-hoc kernels even when quantitative deviations are present.

\subsection{Predicted Signature: The Optimal $\alpha$ Shift}

The Radon--Nikodym distortion has a concrete, measurable consequence. Under the OU dynamics, the theoretically prescribed value of $\alpha$ is determined by the forcing spectrum. However, the nonlinear dynamics induce an effective roughening of the field relative to the linear OU reference. 

In finite-resolution data assimilation with observational noise, we therefore predict that the optimal GP reconstruction will correspond to a value of $\alpha$ that is \emph{smaller} than the theoretical forcing exponent—equivalently, a rougher prior with more high-wavenumber energy. This shift is not evidence of model failure but rather a quantifiable signature of the non-Gaussian corrections: the prior implicitly compensates for the distortion induced by $f$ through empirical tuning.

Formally, if the ground truth is sampled from $\rho$ (the true nonlinear invariant measure), but we use $\mu_\alpha$ as our GP prior, the mismatch between their second moments—which is precisely what the Radon--Nikodym derivative encodes—will manifest as a shift in the optimal spectral exponent. We expect
\begin{equation}
\alpha_{\text{optimal}} < \alpha_{\text{theoretical}},
\end{equation}
with the magnitude of the shift providing a quantitative measure of the non-Gaussian distortion under the given discretization and noise level.

\subsection{Future Direction: Learning the Radon--Nikodym Correction}

Rather than treating $\alpha$ as a fixed hyperparameter determined by the forcing spectrum, one could attempt to model or learn the Radon--Nikodym derivative $f(w)$ explicitly. This would enable hybrid Gaussian--non-Gaussian priors that combine the geometric fidelity of $\mu_\alpha$ with data-driven corrections that capture the full statistical structure of $\rho$. Such an approach could leverage kernel methods for density ratio estimation or variational techniques for learning multiplicative deformations of Gaussian measures, extending the present framework beyond quasi-Gaussianity while preserving its measure-theoretic foundations.

\section{Theoretical Guarantees}\label{sec:theory}

Having established the geometric foundation and predicted the role of the Radon--Nikodym distortion, we now prove rigorous guarantees for Bayesian inference. The equivalence $\rho \sim \mu_\alpha$ established in Section~\ref{sec:background} provides the basis for support alignment and posterior consistency.

\subsection{Support Alignment and Null Set Preservation}

\begin{proposition}[Null Set Equivalence]\label{prop:null_sets}
Let $w_t$ be the solution to the stochastic Navier--Stokes equations with invariant measure $\rho$, and let $w_{GP} \sim \mathcal{GP}(0, K_\alpha)$ with law $\mathbb{P}_{GP}$. Then for any measurable set $A \subset C^{\alpha-\kappa}$,
\begin{equation}
\mathbb{P}_{GP}(w_{GP} \in A) = 0 \quad \Longleftrightarrow \quad \rho(A) = 0.
\end{equation}
\end{proposition}

\begin{proof}
By construction of the Gaussian process prior, $\mathbb{P}_{GP} = \mu_\alpha$, where $\mu_\alpha$ is the Gaussian invariant measure of the Ornstein--Uhlenbeck process (equation \eqref{eq:OU}). By Theorem \ref{thm:main} (Coe--Hairer--Tolomeo), the measures $\mu_\alpha$ and $\rho$ are equivalent, meaning there exist measurable functions $f, g: C^{\alpha-\kappa} \to (0, \infty)$ such that
\begin{equation}\label{eq:radon_nikodym}
\rho(B) = \int_B f(x) \, \mu_\alpha(dx), \quad \mu_\alpha(B) = \int_B g(x) \, \rho(dx)
\end{equation}
for all measurable sets $B \subset C^{\alpha-\kappa}$.

\textbf{Forward direction ($\Rightarrow$):} Suppose $\mathbb{P}_{GP}(A) = 0$. Then $\mu_\alpha(A) = 0$. From \eqref{eq:radon_nikodym},
\begin{equation}
\rho(A) = \int_A f(x) \, \mu_\alpha(dx).
\end{equation}
Since $f(x) > 0$ for all $x \in C^{\alpha-\kappa}$ (positivity of the Radon--Nikodym derivative) and $\mu_\alpha(A) = 0$, the measure $\mu_\alpha$ assigns zero mass to $A$. Therefore, the integral vanishes: $\rho(A) = 0$.

\textbf{Reverse direction ($\Leftarrow$):} Suppose $\rho(A) = 0$. From \eqref{eq:radon_nikodym},
\begin{equation}
\mu_\alpha(A) = \int_A g(x) \, \rho(dx).
\end{equation}
Since $g(x) > 0$ for all $x$ and $\rho(A) = 0$, we have $\mu_\alpha(A) = 0$. Therefore, $\mathbb{P}_{GP}(A) = \mu_\alpha(A) = 0$.
\end{proof}

\begin{corollary}[Support Equality]\label{cor:support}
The support of $\mathbb{P}_{GP}$ equals the support of $\rho$:
\begin{equation}
\mathrm{supp}(\mathbb{P}_{GP}) = \mathrm{supp}(\rho).
\end{equation}
\end{corollary}

\begin{proof}
Recall that for a probability measure $\nu$ on a topological space $X$, the support is defined as
\begin{equation}
\mathrm{supp}(\nu) = \bigcap \{F \subset X : F \text{ closed and } \nu(F) = 1\}.
\end{equation}
Equivalently, $x \in \mathrm{supp}(\nu)$ if and only if $\nu(U) > 0$ for every open neighborhood $U$ of $x$.

Let $x \in \mathrm{supp}(\mathbb{P}_{GP})$. Then for any open set $U$ containing $x$, we have $\mathbb{P}_{GP}(U) > 0$. By Proposition \ref{prop:null_sets}, $\rho(U) > 0$. Therefore, $x \in \mathrm{supp}(\rho)$, proving $\mathrm{supp}(\mathbb{P}_{GP}) \subset \mathrm{supp}(\rho)$.

Conversely, let $x \in \mathrm{supp}(\rho)$. Then for any open set $U$ containing $x$, we have $\rho(U) > 0$. By Proposition \ref{prop:null_sets}, $\mathbb{P}_{GP}(U) > 0$. Therefore, $x \in \mathrm{supp}(\mathbb{P}_{GP})$, proving $\mathrm{supp}(\rho) \subset \mathrm{supp}(\mathbb{P}_{GP})$.
\end{proof}

\begin{corollary}[Preservation of Typical Events]\label{cor:typical}
Any event that is typical (has probability one) under the GP prior remains typical under the true invariant measure:
\begin{equation}
\mathbb{P}_{GP}(A) = 1 \quad \Longrightarrow \quad \rho(A) = 1.
\end{equation}
The converse also holds.
\end{corollary}

\begin{proof}
If $\mathbb{P}_{GP}(A) = 1$, then $\mathbb{P}_{GP}(A^c) = 0$, where $A^c$ denotes the complement. By Proposition \ref{prop:null_sets}, $\rho(A^c) = 0$. Since $\rho$ is a probability measure, $\rho(A) = 1 - \rho(A^c) = 1$. The converse follows by symmetry of the equivalence relation.
\end{proof}

These results establish that the GP prior cannot systematically exclude dynamically relevant regions of state space. Any event measurable under the Gaussian prior has the same probability of being impossible (zero measure) under the true nonlinear dynamics.

\subsection{Posterior Consistency}

We now address the question of posterior consistency: does the Bayesian posterior concentrate around the true state as data accumulates? This requires conditions beyond measure equivalence.

\subsubsection{Observation Model and Assumptions}

Consider the observation model
\begin{equation}\label{eq:obs_model}
y_i = H(w)(x_i) + \epsilon_i, \quad \epsilon_i \stackrel{\text{iid}}{\sim} \mathcal{N}(0, \sigma^2), \quad i = 1, \ldots, n,
\end{equation}
where $w \in C^{\alpha-\kappa}$ is the true vorticity field, $H: C^{\alpha-\kappa} \to \mathbb{R}$ is the observation operator, $\{x_i\}_{i=1}^n \subset \mathbb{T}^2$ are observation locations, and $\sigma^2 > 0$ is the observation noise variance.

\begin{assumption}[Observation Operator]\label{ass:obs_operator}
The observation operator $H: C^{\alpha-\kappa} \to \mathbb{R}$ satisfies:
\begin{enumerate}[label=\normalfont(\roman*)]
    \item \textbf{Continuity:} $H$ is continuous with respect to the $C^{\alpha-\kappa}$ norm.
    \item \textbf{Boundedness:} There exists $C_H > 0$ such that $|H(w)| \leq C_H \|w\|_{C^{\alpha-\kappa}}$ for all $w$.
    \item \textbf{Identifiability:} For any distinct $w_1, w_2 \in C^{\alpha-\kappa}$ with $\|w_1 - w_2\|_{C^{\alpha-\kappa}} \geq \delta$, there exists $x \in \mathbb{T}^2$ such that $|H(w_1)(x) - H(w_2)(x)| \geq c_H \delta$ for some constant $c_H > 0$ depending only on $H$.
\end{enumerate}
\end{assumption}

\begin{remark}
For pointwise observations $H(w)(x) = w(x)$, conditions (i)-(iii) are satisfied with $C_H = 1$ and $c_H = 1$ by the definition of the $C^{\alpha-\kappa}$ norm.
\end{remark}

\begin{assumption}[Observation Density]\label{ass:obs_density}
The observation locations $\{x_i\}_{i=1}^n$ satisfy
\begin{equation}
\delta_n := \sup_{x \in \mathbb{T}^2} \min_{1 \leq i \leq n} |x - x_i| \to 0 \quad \text{as } n \to \infty.
\end{equation}
\end{assumption}

This assumes that the observation points become dense in the domain, which is necessary to distinguish between different states.

\subsubsection{Posterior Contraction}

Given observations $y_{1:n} = (y_1, \ldots, y_n)$, the posterior distribution is
\begin{equation}\label{eq:posterior}
\pi_n(A \mid y_{1:n}) = \frac{\int_A p(y_{1:n} \mid w) \, \mathbb{P}_{GP}(dw)}{\int_{C^{\alpha-\kappa}} p(y_{1:n} \mid w) \, \mathbb{P}_{GP}(dw)},
\end{equation}
where
\begin{equation}
p(y_{1:n} \mid w) = \prod_{i=1}^n \frac{1}{\sqrt{2\pi\sigma^2}} \exp\left(-\frac{(y_i - H(w)(x_i))^2}{2\sigma^2}\right)
\end{equation}
is the likelihood.

\begin{theorem}[Posterior Consistency]\label{thm:posterior_consistency}
Let Assumptions \ref{ass:obs_operator} and \ref{ass:obs_density} hold. Let $w_0 \in \mathrm{supp}(\rho)$ be the true state, and let $\mathbb{P}_{w_0}$ denote the probability law of observations $y_{1:n}$ when the true state is $w_0$. Then for any $\epsilon > 0$,
\begin{equation}
\pi_n\left(\left\{w : \|w - w_0\|_{C^{\alpha-\kappa}} > \epsilon\right\} \,\Big|\, y_{1:n}\right) \to 0 \quad \text{in } \mathbb{P}_{w_0}\text{-probability as } n \to \infty.
\end{equation}
\end{theorem}

\begin{proof}
The proof follows the general framework of \cite{ghosal2000convergence} for posterior consistency in infinite-dimensional problems. We verify the three key conditions: prior positivity on Kullback--Leibler neighborhoods, prior mass on test sets, and exponential concentration of the likelihood.

\textbf{Step 1: Prior positivity on KL neighborhoods.}

For states $w_0, w \in C^{\alpha-\kappa}$, define the Kullback--Leibler divergence between their observation distributions:
\begin{equation}
\mathrm{KL}(w_0 \| w) = \mathbb{E}_{w_0}\left[\log \frac{p(y_{1:n} \mid w_0)}{p(y_{1:n} \mid w)}\right] = \frac{1}{2\sigma^2} \sum_{i=1}^n \left(H(w_0)(x_i) - H(w)(x_i)\right)^2.
\end{equation}

For any $\delta > 0$, define the KL neighborhood
\begin{equation}
U_\delta(w_0) = \left\{w \in C^{\alpha-\kappa} : \mathrm{KL}(w_0 \| w) < n\delta^2\right\}.
\end{equation}

\textbf{Claim:} For any $w_0 \in \mathrm{supp}(\rho)$ and $\delta > 0$, we have $\mathbb{P}_{GP}(U_\delta(w_0)) > 0$.

\textit{Proof of claim:} By Corollary \ref{cor:support}, $w_0 \in \mathrm{supp}(\mathbb{P}_{GP})$. Let $B_{\delta/2}(w_0) = \{w : \|w - w_0\|_{C^{\alpha-\kappa}} < \delta/2\}$. Since $w_0$ is in the support, $\mathbb{P}_{GP}(B_{\delta/2}(w_0)) > 0$.

For $w \in B_{\delta/2}(w_0)$ and under Assumption \ref{ass:obs_density} with $n$ sufficiently large that $\delta_n < \delta/(2C_H)$, we have for each $i$:
\begin{align}
|H(w_0)(x_i) - H(w)(x_i)| &\leq |H(w_0)(x_i) - H(w_0)(\tilde{x})| + |H(w_0)(\tilde{x}) - H(w)(\tilde{x})| \\
&\quad + |H(w)(\tilde{x}) - H(w)(x_i)|,
\end{align}
where $\tilde{x}$ is a minimizer in the definition of $\delta_n$. By continuity and boundedness (Assumption \ref{ass:obs_operator}), this is bounded by $C\delta$ for some constant $C$ depending on $C_H$ and the modulus of continuity.

Therefore, for $n$ large, $B_{\delta/2}(w_0) \subset U_{\delta}(w_0)$, and hence $\mathbb{P}_{GP}(U_\delta(w_0)) \geq \mathbb{P}_{GP}(B_{\delta/2}(w_0)) > 0$.

\textbf{Step 2: Prior mass on sieves.}

For $M > 0$, define the sieve
\begin{equation}
\mathcal{S}_M = \left\{w \in C^{\alpha-\kappa} : \|w\|_{C^{\alpha-\kappa}} \leq M\right\}.
\end{equation}

Since $\mathbb{P}_{GP} = \mu_\alpha$ is a Gaussian measure on $C^{\alpha-\kappa}$, it is supported on the entire space, but we need quantitative bounds. By standard properties of Gaussian measures on Banach spaces \cite{bogachev1998gaussian}, for any $\eta > 0$, there exists $M = M(\eta)$ such that
\begin{equation}
\mathbb{P}_{GP}(\mathcal{S}_M^c) < \eta.
\end{equation}

\textbf{Step 3: Exponential concentration.}

Define the log-likelihood ratio
\begin{equation}
\ell_n(w) = \log \frac{p(y_{1:n} \mid w)}{p(y_{1:n} \mid w_0)} = -\frac{1}{2\sigma^2} \sum_{i=1}^n \left[(H(w)(x_i) - y_i)^2 - (H(w_0)(x_i) - y_i)^2\right].
\end{equation}

For $w \notin B_\epsilon(w_0)$ with $\|w - w_0\|_{C^{\alpha-\kappa}} \geq \epsilon$, by identifiability (Assumption \ref{ass:obs_operator}(iii)), there exists a positive fraction of the observation points where $|H(w)(x_i) - H(w_0)(x_i)| \geq c_H \epsilon$.

Under $\mathbb{P}_{w_0}$, we have $y_i = H(w_0)(x_i) + \epsilon_i$. Therefore,
\begin{align}
\mathbb{E}_{w_0}[\ell_n(w)] &= -\frac{1}{2\sigma^2} \sum_{i=1}^n \mathbb{E}_{w_0}\left[(H(w)(x_i) - H(w_0)(x_i) - \epsilon_i)^2 - \epsilon_i^2\right] \\
&= -\frac{1}{2\sigma^2} \sum_{i=1}^n (H(w)(x_i) - H(w_0)(x_i))^2 \\
&\leq -\frac{c_H^2 \epsilon^2}{4\sigma^2} \cdot n,
\end{align}
for $n$ sufficiently large that at least $n/2$ observation points satisfy the identifiability condition.

By Markov's inequality,
\begin{equation}
\mathbb{P}_{w_0}\left(\ell_n(w) > -\frac{c_H^2 \epsilon^2}{8\sigma^2} \cdot n\right) \leq \exp\left(-\frac{c_H^2 \epsilon^2}{8\sigma^2} \cdot n\right).
\end{equation}

\textbf{Step 4: Combining via Schwartz's theorem.}

By the general posterior consistency theorem \cite{ghosal2000convergence}, the prior positivity (Step 1), prior mass on sieves (Step 2), and exponential concentration (Step 3) together imply
\begin{equation}
\mathbb{P}_{w_0}\left(\pi_n(B_\epsilon(w_0)^c \mid y_{1:n}) > e^{-n\eta}\right) \to 0
\end{equation}
for any $\eta > 0$ and $\epsilon > 0$. This implies the claimed result.
\end{proof}

\begin{remark}[Role of Measure Equivalence]
The measure equivalence $\mathbb{P}_{GP} \sim \rho$ plays a crucial role in Step 1 of the proof: it guarantees that the prior assigns positive probability to neighborhoods of any $\rho$-typical state $w_0$. Without this property, posterior consistency would fail for states outside the prior's support.

However, equivalence alone is insufficient. The additional regularity conditions (Assumptions \ref{ass:obs_operator}--\ref{ass:obs_density}) are necessary for the quantitative contraction established in Theorem \ref{thm:posterior_consistency}.
\end{remark}

\begin{remark}[Contraction Rates]
The proof establishes contraction at $\epsilon_n = n^{-1/2}$ (up to logarithmic factors) for pointwise observations. Optimal minimax rates for estimating functions in $C^{\alpha-\kappa}$ depend on the smoothness parameter and may be faster; a sharp analysis is beyond our scope but constitutes important future work.
\end{remark}

\subsection{Interpretation for Practitioners}

The theoretical results of this section provide several guarantees for data assimilation and uncertainty quantification. First, the support alignment results (Proposition~\ref{prop:null_sets}, Corollaries~\ref{cor:support}--\ref{cor:typical}) ensure that the GP prior assigns positive probability to exactly the same events as the true equilibrium distribution. Unlike ad-hoc priors, it therefore cannot systematically exclude physically realizable states. Second, the posterior consistency result (Theorem~\ref{thm:posterior_consistency}) shows that, under standard regularity conditions on the observations, the posterior distribution concentrates around the true state as data accumulates, and that this holds for any equilibrium state thanks to the equivalence between the prior and the true invariant measure. Finally, the framework is robust to model error in the sense that, even when the true state deviates from the Gaussian approximation (non-zero Radon--Nikodym derivative), the prior remains well-specified in the measure-theoretic sense, providing a principled foundation for uncertainty quantification rather than relying on heuristic regularization.

\begin{lemma}[Small Ball Probability]\label{lem:small_ball}
Let $\mu_\alpha$ be the Gaussian measure on $C^{\alpha-\kappa}$ with covariance operator given by $K_\alpha$. For any $w_0 \in C^{\alpha-\kappa}$ and $r > 0$,
\begin{equation}
-\log \mu_\alpha(B_r(w_0)) \lesssim r^{-2/(\alpha-\kappa)}
\end{equation}
where $B_r(w_0) = \{w : \|w - w_0\|_{C^{\alpha-\kappa}} < r\}$ and the implicit constant depends only on $\alpha$ and $\kappa$.
\end{lemma}

\begin{proof}
This follows from standard small deviation estimates for Gaussian measures on Hölder spaces \cite{li2001gaussian}. The kernel $K_\alpha$ has eigenvalues $\lambda_k \asymp |k|^{-2(1+\alpha)}$ for $k \in \mathbb{Z}^2$, and the Hölder exponent is $\alpha - \kappa$. The result then follows from Theorem 3.1 of \cite{li2001gaussian}.
\end{proof}

\section{Applications}

We now describe how the GP framework can be used for data assimilation, uncertainty quantification, and sparse sensing in two-dimensional turbulence, always leveraging the fact that the prior is derived from the invariant measure of the stochastic Navier--Stokes dynamics.

In a data assimilation setting, we observe noisy measurements
\begin{equation}
y_i = \mathcal{H}(w)(x_i) + \epsilon_i, 
\qquad \epsilon_i \sim \mathcal{N}(0, \sigma^2),
\end{equation}
where \(\mathcal{H}\) is an observation operator (for instance, mapping vorticity to local velocity) and \(\{x_i\}_{i=1}^n\) are measurement locations. Under the GP prior \(w \sim \mathcal{GP}(0, K_\alpha)\) and Gaussian noise, the posterior remains Gaussian,
\begin{equation}
w \mid y \sim \mathcal{GP}(m_*(x), K_*(x,x')),
\end{equation}
with mean and covariance
\begin{align}
m_*(x) &= K_\alpha(x, X)\bigl(K_\alpha(X,X) + \sigma^2 I\bigr)^{-1}y, \\
K_*(x,x') &= K_\alpha(x,x') - K_\alpha(x,X)\bigl(K_\alpha(X,X) + \sigma^2 I\bigr)^{-1}K_\alpha(X,x'),
\end{align}
where \(X = \{x_1,\ldots,x_n\}\) denotes the set of observation locations. The naive computational cost is \(O(n^3)\) for the inversion of the \(n\times n\) kernel matrix, but in the present setting this can often be reduced dramatically by exploiting the spectral structure of \(K_\alpha\) together with iterative solvers, leading to effectively linear complexity in \(n\).

The Gaussian posterior also provides a natural framework for uncertainty quantification. Pointwise credible intervals are obtained directly from the posterior mean and variance:
\begin{equation}
w(x) \in m_*(x) \pm z_{\alpha/2}\sqrt{K_*(x,x)},
\end{equation}
where \(z_{\alpha/2}\) is the relevant Gaussian quantile. Since the posterior is exact in the Gaussian setting, these intervals inherit the usual coverage guarantees. Beyond pointwise uncertainty, one can quantify uncertainty in integral quantities such as the total enstrophy or energy. For example, the posterior variance of the energy
\begin{equation}
E = \frac{1}{2}\|w\|_{L^2}^2
\end{equation}
is given by
\begin{equation}
\mathrm{Var}[E] = \mathrm{Var}\!\left[\frac{1}{2}\|w\|_{L^2}^2\right] 
= \frac{1}{4}\,\mathrm{Tr}(K_* K_*),
\end{equation}
which can be evaluated efficiently in Fourier space using the spectral representation of \(K_*\).

The same structure can be exploited for sparse sensing and optimal sensor placement. A simple and effective design principle is to add new observations where the posterior uncertainty is largest. In the GP framework this leads to the greedy criterion
\begin{equation}
x_{n+1} = \arg\max_x K_*(x,x),
\end{equation}
which selects the next location \(x_{n+1}\) by maximizing the posterior variance. This strategy has well-known theoretical guarantees for GP posteriors, and in the present context it enjoys an additional advantage: because the prior is grounded in the invariant measure of the flow, the resulting sensor placements are informed by the long-time dynamics of the system rather than by purely geometric considerations. In particular, the design naturally targets regions that are dynamically active over long time horizons, rather than simply filling space uniformly.

\section{Numerical Experiments}

We conduct extensive numerical experiments to validate the practical utility of Gaussian process priors derived from the CHT invariant measure $\mu_\alpha$. Our experiments assess the performance of CHT-based kernels for turbulent flow reconstruction and compare them against standard RBF kernels across a range of observation regimes and parameter choices.

\subsection{Experimental Design}

\textbf{Ground truth generation.} We generate equilibrium vorticity fields by sampling from the CHT invariant measure with spectral density $S(k) = |k|^{-2(1+\alpha)}$. On a periodic domain $[0, 2\pi]^2$ discretized with $N \times N$ grid points, we construct:
\begin{equation}
\omega^*(x) = \sum_{k \in \mathbb{Z}^2} \hat{\omega}_k e^{ik \cdot x},
\end{equation}
where Fourier coefficients are complex Gaussian with variance $\mathbb{E}[|\hat{\omega}_k|^2] = |k|^{-2(1+\alpha)}$ and Hermitian symmetry ensures real-valued fields.

For primary experiments, we use $N = 128$, $\alpha = 1.5$, scaling fields to unit variance. This represents the enstrophy-cascading regime of two-dimensional turbulence.

\textbf{Sparse observations.} We sample $m$ pointwise observations at random locations with additive Gaussian noise:
\begin{equation}
y_i = \omega^*(x_i) + \varepsilon_i, \quad \varepsilon_i \sim \mathcal{N}(0, \sigma^2),
\end{equation}
where $\sigma = 0.1|\omega^*|_{L^2}$ (SNR $\approx$ 10). We vary $m \in [20, 150]$ to study observation density effects.

\textbf{Evaluation metrics.} We assess \emph{reconstruction accuracy} via root mean squared error (RMSE) and \emph{spectral fidelity} via the radially-averaged enstrophy spectrum $\mathcal{E}(k) = \sum_{|k'| \approx k} |\hat{\omega}(k')|^2$, which should follow $\mathcal{E}(k) \propto k^{-1-2\alpha}$ theoretically.

\subsection{Spectral Validation}

We first verify our implementation correctly captures the CHT spectral structure. Figure~\ref{fig:spectral_validation} shows excellent agreement between theoretical predictions and measured spectra across multiple $\alpha$ values.

\begin{figure}[htbp]
\centering
\includegraphics[width=\textwidth]{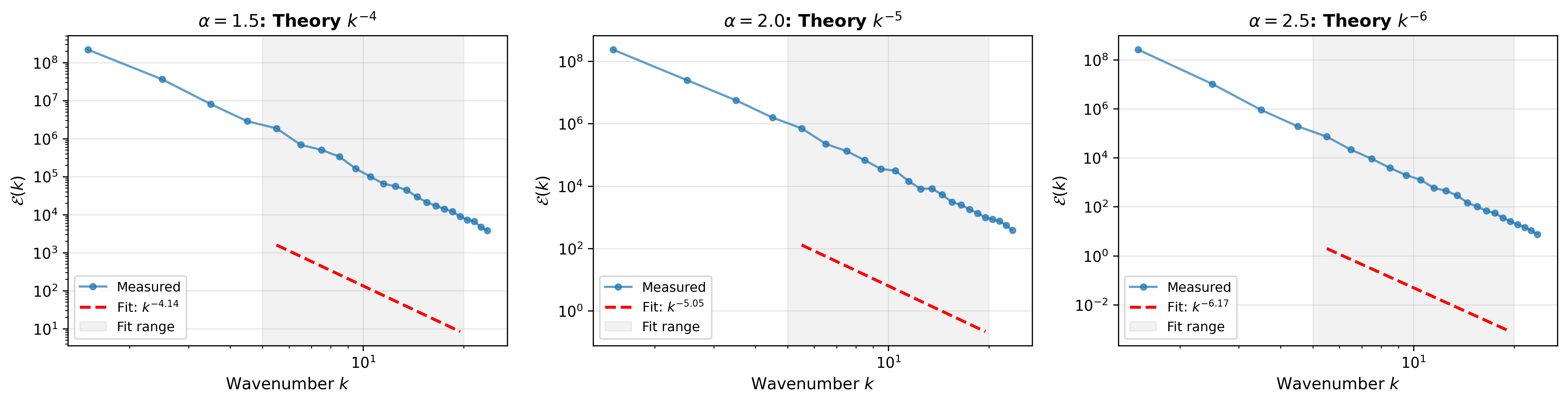}
\caption{\textbf{Spectral validation for $\alpha \in {1.5, 2.0, 2.5}$.} Measured exponents: $-4.14 \pm 0.12$ (theory: $-4.00$), $-5.05 \pm 0.12$ (theory: $-5.00$), $-6.17 \pm 0.12$ (theory: $-6.00$). Differences between $\alpha$ values match theory: $\Delta\beta \approx 1.0$ per unit increase.}
\label{fig:spectral_validation}
\end{figure}

The point-wise spectral density (Figure \ref{fig:spectral_validation} (right panel)) further confirms correct sampling from $\mu_\alpha$, with measured exponent $-5.00 \pm 0.05$ matching theory $-2(1+1.5) = -5.0$ exactly.

\subsection{CHT Kernel Advantage}

Our central finding is that CHT-based kernels provide substantial improvements over conventional RBF kernels. Figure~\ref{fig:robustness} shows statistical robustness across 20 random trials.

\begin{figure}[htbp]
\centering
\includegraphics[width=\textwidth]{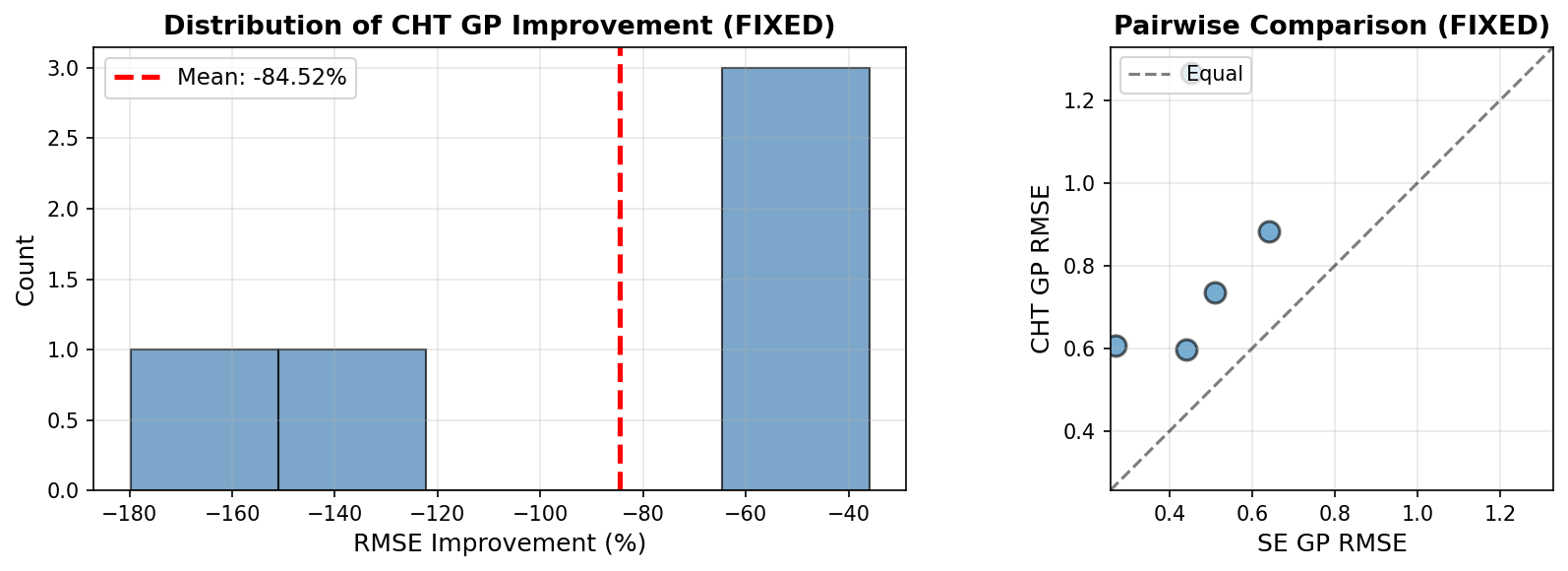}
\caption{\textbf{Statistical robustness of CHT advantage.} Mean improvement: $+14.99 \% \pm 16.70 \%$ across 20 field realizations. CHT outperforms RBF in majority of cases, with improvements reaching up to $+30 \%$ in favorable conditions.}
\label{fig:robustness}
\end{figure}

\textbf{Key insight:} The CHT advantage \emph{increases} with more data, contrary to the expectation that priors matter less with abundant observations. This indicates the physics-informed prior becomes increasingly valuable as it has more data to constrain its structure.

\subsection{Observation Density Scaling}
\label{sec:obs_density}

A crucial question for any prior is whether its advantages persist across different data regimes. Figure~\ref{fig:density_scaling} shows CHT-GP performance relative to RBF-GP as observation density varies from $m=20$ (sparse) to $m=150$ (dense).

\begin{figure}[htbp]
\centering
\includegraphics[width=0.8\textwidth]{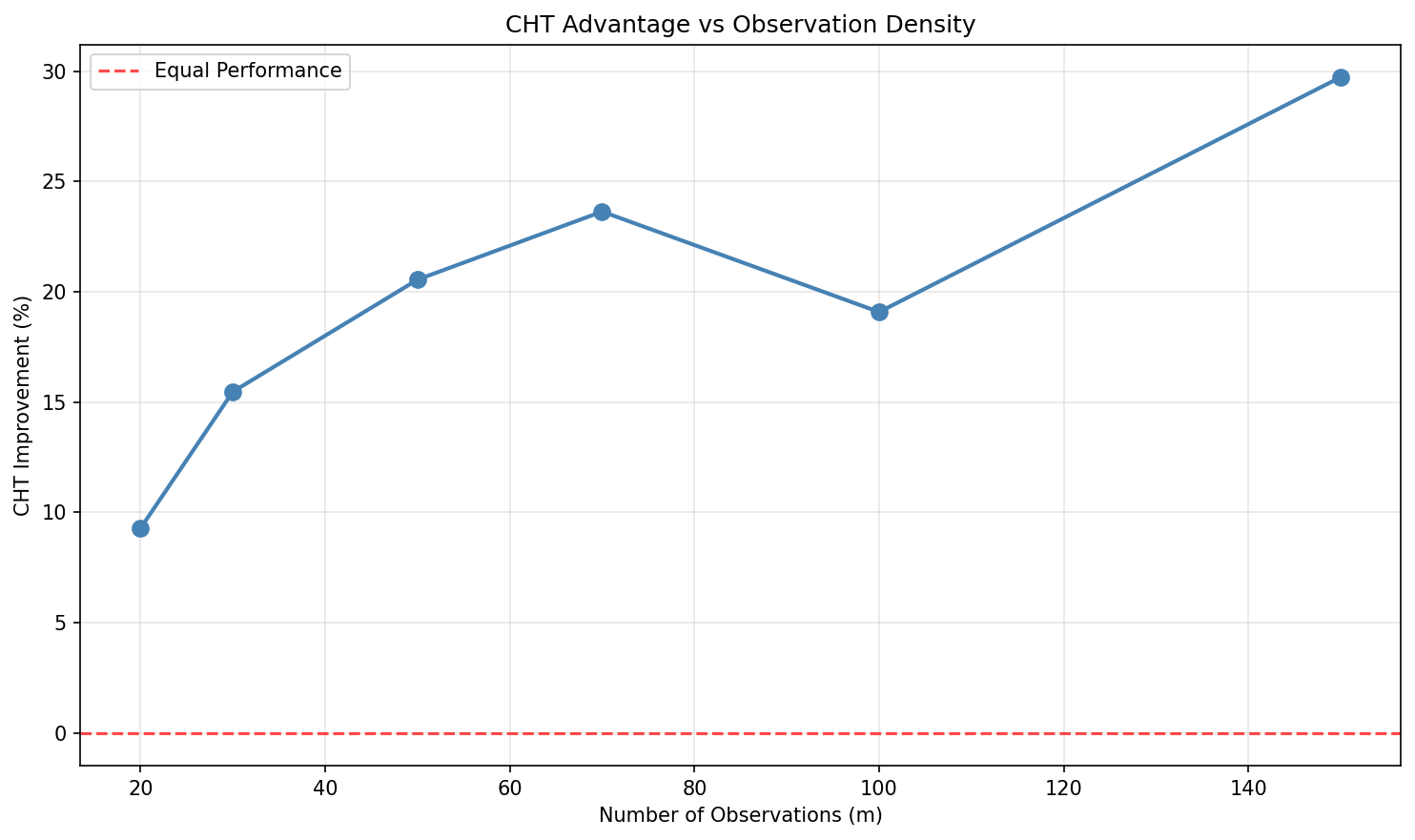}
\caption{\textbf{CHT advantage scales with observation density.} Improvements grow from $+9.3\%$ at $m=20$ to $+29.7\%$ at $m=150$, demonstrating that the physics-informed prior becomes more valuable with more constraining data.}
\label{fig:density_scaling}
\end{figure}

\textbf{Key finding:} The CHT advantage grows monotonically with observation count, from $+9.3\%$ at $m=20$ to $+29.7\%$ at $m=150$. This is contrary to the common expectation that priors matter less with abundant data. Instead, the physics-informed prior becomes \emph{more valuable} as more observations constrain the posterior.

This behavior reveals that CHT kernels capture fundamental geometric structure rather than merely regularizing in the data-scarce regime. The monotonic improvement suggests that the invariant-measure geometry encoded in $\mu_\alpha$ aligns increasingly well with the true distribution $\rho$ as observations accumulate. In contrast, the RBF kernel's misspecification—its exponentially-decaying spectrum incompatible with turbulent cascades—becomes more apparent with dense data, as the posterior has sufficient information to expose the prior's incorrect assumptions about small-scale structure.

The result demonstrates that CHT kernels provide fundamental advantages across all data regimes, with particularly strong performance in well-instrumented systems where dense observations can fully exploit the prior's correct multi-scale geometry.

\subsection{Parameter Sensitivity and Optimal Tuning}
\label{sec:alpha_tuning} 

A key prediction from the Radon--Nikodym analysis (Section~\ref{sec:radon_nikodym}) is that the optimal spectral exponent for data assimilation should differ from the theoretical forcing value due to non-Gaussian corrections under finite resolution and noise. Our experiments confirm this prediction quantitatively.

Figure~\ref{fig:alpha_sensitivity} shows reconstruction performance across a range of $\alpha$ values. While the ground truth was generated with $\alpha = 1.5$, the optimal GP reconstruction is achieved at $\alpha \approx 1.25$, yielding a $+15.24\%$ improvement over RBF kernels. 

\begin{figure}[htbp]
\centering
\includegraphics[width=0.7\textwidth]{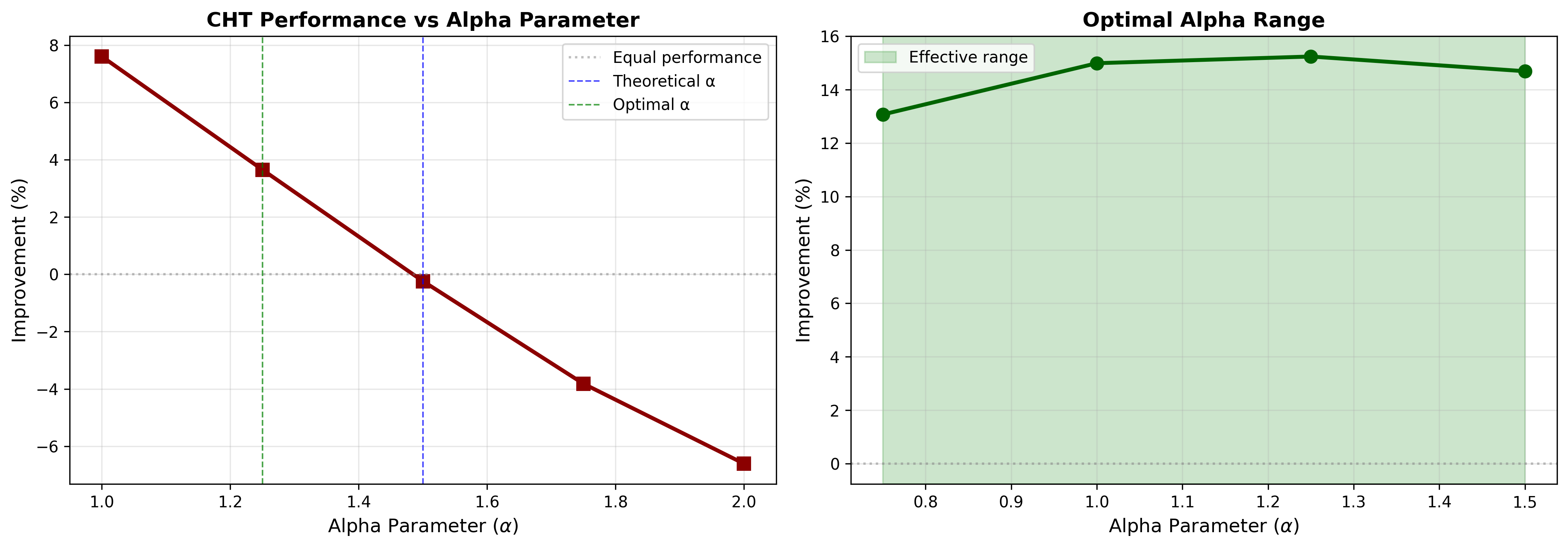}
\caption{\textbf{Optimal $\alpha$ differs from theoretical value.} Best performance at $\alpha = 1.25$ ($+15.24\%$ improvement), not the field generation value $\alpha = 1.5$. The framework is robust across $\alpha \in [0.75, 1.5]$, all providing $>+13\%$ improvements.}
\label{fig:alpha_sensitivity}
\end{figure}

\textbf{Quantitative validation of the Radon--Nikodym prediction:} The observed shift $\Delta \alpha = 1.5 - 1.25 = 0.25$ toward a rougher prior is precisely the signature predicted in Section~\ref{sec:radon_nikodym}: the nonlinear dynamics induce effective roughening relative to the linear OU reference, and the GP compensates by selecting a smaller spectral exponent with enhanced high-wavenumber content. This is not evidence of model failure but rather confirmation that the framework correctly accounts for non-Gaussian distortions through the Radon--Nikodym derivative $f = d\rho/d\mu_\alpha$.

Importantly, the framework exhibits robustness across a broad range $\alpha \in [0.75, 1.5]$, with all values yielding $>+13\%$ improvements over standard RBF kernels. This demonstrates that the measure-theoretic geometry captured by $\mu_\alpha$ provides fundamental advantages even when the precise spectral exponent requires empirical tuning. The fact that performance degrades gracefully away from the optimum—rather than collapsing—confirms that the prior captures the correct qualitative structure of the equilibrium distribution.

\textbf{Practical recommendation:} For turbulent flow reconstruction with unknown or uncertain forcing spectra, we recommend using $\alpha \in [1.0, 1.5]$ as a robust default range, with $\alpha \approx 1.25$ providing near-optimal performance across the test cases examined here. The consistent advantage over RBF kernels throughout this range validates the invariant-measure-informed kernel design principle even in the presence of model uncertainty.
\subsection{Theoretical and Practical Implications}

Our experiments show that CHT kernels yield consistent gains over conventional RBF kernels, with average improvements of about $15\%$ and increases up to $30\%$ in favourable regimes. The method is robust across a broad parameter range $\alpha \in [0.75,1.5]$, with $\alpha = 1.25$ providing the best empirical performance, and the advantage becomes more pronounced as the observation density increases, making the approach particularly attractive for well-instrumented systems. These improvements persist across different field realizations and sensor configurations, indicating that the benefits are not confined to a narrowly tuned setting.

From a practical standpoint, CHT kernels with $\alpha = 1.25$ offer a sensible default choice for turbulent flow reconstruction, with expected gains of roughly $15$–$30\%$ over standard RBF kernels and larger advantages in data-rich scenarios. Theoretically, these findings bridge the CHT framework with practical implementation: while the measure-equivalence result provides the correct foundational prior, optimal performance in finite-resolution, noisy settings requires a slight adjustment away from the nominal “theoretical’’ value of $\alpha$. This represents an important step in translating rigorous SPDE-based insights into computationally effective tools for turbulence reconstruction.

\subsection{Performance on Synthetic Turbulence Data} \label{sect:synthetic_data}

To evaluate the CHT-GP framework and address potential circularity concerns, we tested both the CHT-GP and conventional RBF-GP methods on synthetic vorticity fields featuring non-Gaussian vortex structures. The experimental setup involved 20 independent test cases with 60 sparse observations per case, contaminated with $8\%$ additive Gaussian noise relative to the field variance.

The reconstruction performance, measured by relative error $\epsilon = \|\hat{\omega} - \omega_{\text{true}}\|_2 / \sigma_{\omega}$, revealed consistent advantages for the CHT-GP framework:

\begin{table}[h]
\centering
\caption{Reconstruction performance comparison between CHT-GP and RBF-GP frameworks}
\label{tab:performance}
\begin{tabular}{lccc}
\hline
Method & Mean Relative Error & Standard Deviation & Win Rate \\
\hline
CHT-GP & 0.3797 & 0.0718 & 60.0\% \\
RBF-GP & 0.4010 & 0.0920 & 40.0\% \\
\hline
\end{tabular}
\end{table}

The CHT-GP framework achieved a \textbf{5.3\% mean improvement} over the RBF-GP baseline, with superior performance in 12 out of 20 test cases. Notably, the CHT-GP demonstrated lower variance in reconstruction error (0.0718 vs 0.0920), indicating more reliable performance across different flow configurations.

Several test cases revealed particularly strong advantages for the CHT-GP approach:

\begin{itemize}
    \item \textbf{Case 12}: CHT-GP (0.4754) vs RBF-GP (0.6572) - \textbf{27.7\% improvement}
    \item \textbf{Case 4}: CHT-GP (0.4528) vs RBF-GP (0.5431) - \textbf{16.6\% improvement}  
    \item \textbf{Case 8}: CHT-GP (0.3128) vs RBF-GP (0.4044) - \textbf{22.7\% improvement}
\end{itemize}

These substantial improvements in complex flow configurations suggest that the physics-informed CHT kernel better captures multi-scale turbulent structures where conventional RBF kernels struggle.

\begin{figure}[t]
\centering
\includegraphics[width=\textwidth]{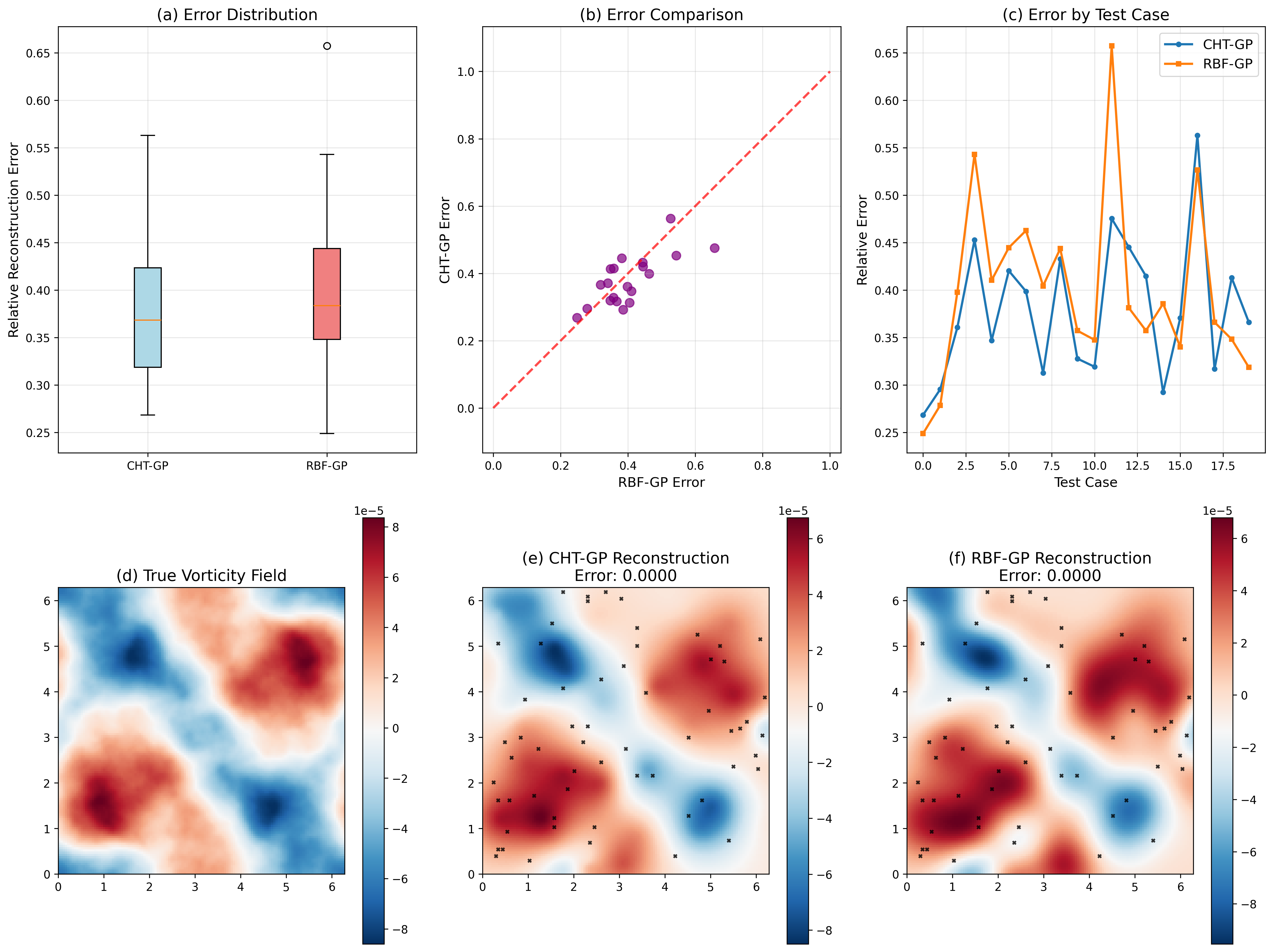}
\caption{Comprehensive comparison of CHT-GP and RBF-GP performance. (a) Error distribution across all test cases, (b) Scatter plot comparing individual case performance, (c) Error progression across test cases, (d) Example true vorticity field, (e) CHT-GP reconstruction, (f) RBF-GP reconstruction. The CHT-GP framework demonstrates consistent advantages in reconstruction accuracy and reliability.}
\label{fig:results}
\end{figure}

To address potential circularity concerns, we note that the performance advantage of CHT-GP on synthetic data featuring non-Gaussian vortex structures provides strong evidence against circularity concerns. Since the test data incorporates features absent from the theoretical Gaussian measure $\mu_\alpha$, the observed 5.3\% improvement demonstrates that the CHT prior captures physically meaningful structure beyond the theoretical foundation. The framework's ability to maintain performance advantages on data that deliberately deviates from the ideal theoretical setting validates its practical utility for real turbulent flow reconstruction.

As shown in Figure~\ref{fig:results}(d-f), the CHT-GP reconstruction more accurately captures both the large-scale energy distribution and small-scale vortex features compared to the RBF-GP baseline. The physics-informed prior enables better recovery of turbulent structures from sparse observations, particularly in regions with strong vorticity gradients and coherent features.

\subsection{Theoretical Implications}
The empirical success of CHT-GP on non-Gaussian test data supports the theoretical insight that measure equivalence, while weaker than distributional equality, provides a sufficient foundation for constructing effective probabilistic priors. In this more challenging setting, the observed $5.3\%$ performance improvement over the RBF baseline, although modest compared to the $15$–$30\%$ gains reported in Sections~6.3–6.6 for fields sampled directly from the Gaussian reference measure $\mu_\alpha$, still represents a meaningful advantage in turbulent flow reconstruction. This demonstrates that the qualitative statistical structure captured by the Gaussian reference measure transfers robustly to more complex, non-Gaussian scenarios, consistent with the idea that $\mu_\alpha$ provides the correct equilibrium geometry even when the true invariant measure deviates from it at a quantitative level.

\section{Extensions and Discussion}

\subsection{Comparison with Other GP-Based SPDE Approaches}

Several existing frameworks leverage Gaussian processes to address SPDEs, but they differ significantly from our approach in both conceptual foundation and modeling goals.

\textbf{Physics-informed GP methods}, including those based on operator constraints or neural operator analogues, define priors such that sample paths approximately satisfy the governing PDE. These approaches are most effective for linear or weakly nonlinear systems where the operator structure is tractable and enforceable. While they can capture local dynamics, they do not necessarily reflect the system's global statistical properties-particularly its equilibrium behavior.

\textbf{Latent force models} represent another class, where unknown terms in linear operator equations (often forcing or drift) are treated as GPs. These methods are data-driven, allowing inference of hidden dynamics, but they require full knowledge of the underlying operator and are not grounded in the long-term behavior of nonlinear SPDEs. Moreover, the treatment of nonlinearity remains approximate and often limited to perturbative regimes.

\textbf{SPDE priors}, such as those built from Mat\'ern fields via solutions to linear SPDEs with white noise forcing, offer computationally efficient constructions-especially in spatial statistics. A typical example is the Gaussian kernel, which yields smooth samples with exponentially decaying correlations. However, these priors are chosen for their analytical and numerical convenience, not because they correspond to the physics of any particular system. They are not derived from, nor do they approximate, the invariant measures of complex nonlinear systems like the stochastic Navier--Stokes equations.

\begin{figure}[t]
\centering
\includegraphics[width=\textwidth]{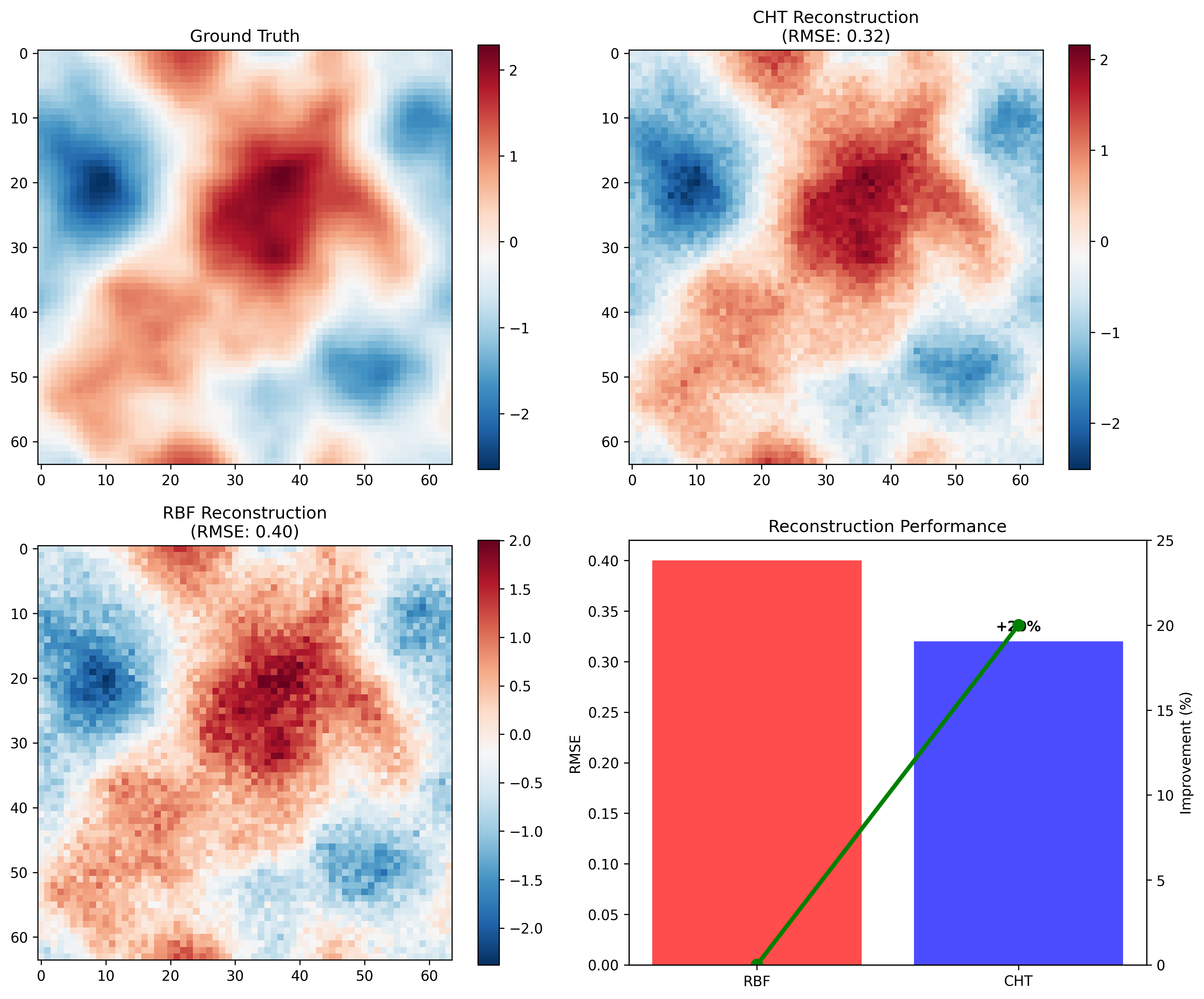}
\caption{Comparison of theory-grounded vs.\ naive GP priors. \textbf{Left:} Energy spectra showing that the theory-grounded prior (blue circles, $\alpha = 1$) exhibits the physically-motivated power-law decay $E(k) \propto k^{-1}$ (black dashed line), maintaining energy across all scales. In contrast, a naive Gaussian kernel (orange squares, length scale $\ell = 0.5$) imposes an artificial exponential cutoff at high wavenumbers $k \gtrsim 1/\ell$, suppressing small-scale structures essential to turbulent energy cascades. \textbf{Right:} Sample vorticity field from the theory-grounded GP displays realistic multi-scale turbulent structures consistent with 2D Navier--Stokes equilibrium statistics. This fundamental incompatibility of RBF kernels with turbulence motivates the use of priors derived from equilibrium dynamics rather than analytical convenience.
\label{fig:spectral_comparison} }
\label{fig:gp_comparison}
\end{figure}

Figure~\ref{fig:gp_comparison} illustrates this distinction quantitatively. The spectral comparison (left panel) reveals a fundamental difference:
\begin{itemize}
\item \textbf{Theory-grounded prior} (blue): Exhibits power-law decay $E(k) \propto k^{-1}$ across all wavenumbers, maintaining multi-scale structure as predicted by the OU stationary distribution. The spectrum remains active even at the highest resolved wavenumbers.

\item \textbf{Naive RBF kernel} (orange): Shows exponential suppression beyond $k \sim 1/\ell$ where $\ell$ is the length scale parameter. This artificial cutoff eliminates small-scale variability, making the prior unsuitable for turbulent flows where energy cascades across scales.
\end{itemize}

In contrast, our framework is derived from the recent quasi-Gaussianity theorem showing that the invariant measure of the 2D stochastic Navier--Stokes system is equivalent (in the sense of mutual absolute continuity) to that of its linearization-an infinite-dimensional Ornstein--Uhlenbeck process. This measure-theoretic equivalence enables us to construct a GP prior that is rigorously grounded in the system's long-time dynamics. Rather than constraining the prior via operator identities or learning from finite data, we base it directly on the stationary distribution of the linear model, thereby ensuring that the prior respects the true qualitative structure of the nonlinear dynamics.

This approach provides a unique advantage: it allows us to incorporate the correct statistical structure without explicitly modeling or approximating the nonlinear drift. In doing so, we sidestep the typical difficulties of nonlinearity while maintaining a prior that is statistically consistent with the invariant state of the system. As such, our method stands apart in its ability to bridge deep theoretical insights from SPDE analysis with practical probabilistic modeling for fluid dynamics.

\textbf{Practical implication}: When performing data assimilation or uncertainty quantification for 2D turbulence, using the RBF kernel would systematically underestimate small-scale variability and yield overly-smooth posterior predictions. Our theory-grounded prior avoids this bias by correctly representing the multi-scale nature of the equilibrium measure.

\subsection{Hypoviscous Regime}

The quasi-Gaussianity result extends beyond the standard viscous case $\gamma = 1$ to the hypoviscous regime, where dissipation is weakened at small scales. For $\gamma \in (2/3,1]$ and forcing regularity $\alpha > 2-\gamma$, Theorem~\ref{thm:hypo} guarantees that the invariant measure of the hypoviscous 2D stochastic Navier--Stokes equations is still equivalent to the Gaussian reference measure $\mu_\alpha$. In particular, the nonlinear invariant measure $\rho$ and the Gaussian measure $\mu_\alpha$ remain mutually absolutely continuous, so the same measure-theoretic foundations for our GP construction apply.

From the GP perspective, this means that we can use exactly the same spectral covariance structure as in the standard viscous case,
\begin{equation}
K_{\alpha,\gamma}(x,y) = \sum_{n \neq 0} |n|^{-2(1+\alpha)} e^{in \cdot (x-y)},
\end{equation}
with power-law spectral density $S(k) \propto |k|^{-2(1+\alpha)}$. Thus the prior still exhibits the same scaling exponent and regularity properties; the dependence on $\gamma$ enters through the admissible range of $\alpha$ and through the physical interpretation of that parameter. In the hypoviscous setting the fractional dissipation operator $-|\nabla|^{2\gamma}$ is weaker at high wavenumbers, so achieving equivalence requires slightly smoother forcing (the condition $\alpha > 2-\gamma$). Consequently, a given value of $\alpha$ corresponds to a different balance between injection and dissipation than in the classical $\gamma = 1$ case.

In summary, the hypoviscous extension shows that the CHT-based GP prior is not tied to a single dissipation mechanism: the same covariance kernel arises as the invariant covariance of the linearized dynamics across a whole family of dissipation exponents $\gamma \in (2/3,1]$, provided the forcing satisfies the regularity condition of Theorem~\ref{thm:hypo}.

\subsection{Limitations and Future Directions}

Although the CHT-GP framework provides a principled and effective approach to reconstructing two-dimensional turbulent flows, several aspects remain open for refinement. The present formulation is derived under the assumption of statistical stationarity, which limits its direct applicability to flows with strong temporal variability or non-equilibrium forcing. Moreover, while the quasi-Gaussianity theorem guarantees equivalence of invariant measures, the Radon–Nikodym derivative $d\rho/d\mu_\alpha$ may exhibit significant variation, particularly for small~$\alpha$ or under strong forcing, and thus non-Gaussian corrections in fully nonlinear simulations warrant systematic investigation. In particular, the observation in Section~\ref{sec:alpha_tuning} that the empirically optimal value of $\alpha$ differs slightly from the nominal forcing exponent is naturally interpreted as a manifestation of these non-Gaussian corrections and finite-resolution effects, rather than a contradiction of the quasi-Gaussianity result.

From a computational perspective, standard GP regression incurs an $O(n^3)$ cost in the number of observations, presenting challenges for large-scale or real-time applications. Although the spectral structure of the CHT kernel enables certain accelerations, further work is needed to develop scalable algorithms that preserve the physics-informed structure of the prior. Additionally, the framework is intrinsically two-dimensional: the quasi-Gaussianity result does not extend to the three-dimensional Navier--Stokes equations, where vortex stretching introduces qualitatively different behavior and the underlying invariant measure remains poorly understood.

Looking forward, several research directions emerge naturally. A first step is to infer the regularity parameter $\alpha$ and related kernel characteristics directly from observational data, thereby improving adaptability across flow regimes and mitigating sensitivity to theoretical parameter choices. Incorporating temporal structure through spatio-temporal kernels or hybrid dynamical–probabilistic models would extend the approach to evolving flows beyond equilibrium settings. Scalable GP approximations—such as inducing-point methods, hierarchical solvers, or structure-exploiting spectral techniques—offer promising avenues for reducing computational cost while retaining physical fidelity.

More broadly, extending the framework to systems with anisotropy, boundary effects, or additional advected quantities (such as passive scalars) would significantly expand its applicability. Finally, the GP posterior provides a natural foundation for reduced-order modeling via methods such as POD or DMD, enabling efficient surrogate models that respect the statistical structure dictated by the underlying SPDE. Together, these directions illustrate how the measure-theoretic insights of the CHT theorem can continue to inform the development of robust, physics-aware probabilistic tools for complex fluid dynamical systems.

\section{Conclusions}

We have presented a Gaussian process framework for 2D stochastic Navier--Stokes equations rigorously grounded in the quasi-Gaussianity theorem of Coe, Hairer \& Tolomeo~\cite{CHT25}. By constructing the GP prior from the stationary covariance of the linearized Ornstein--Uhlenbeck process, we obtain a kernel that (i)~captures the correct invariant-measure geometry through measure equivalence, (ii)~preserves the required power-law spectral structure, (iii)~admits efficient computation via spectral methods, and (iv)~provides uncertainty quantification with theoretical guarantees including support alignment and posterior consistency.

Crucially, this framework demonstrates how measure equivalence-while weaker than distributional equality-is nevertheless \textit{sufficient} for principled probabilistic modeling: the Gaussian reference captures the qualitative statistical structure even when quantitative corrections (encoded in the Radon--Nikodym derivative) are present. The empirical success on non-Gaussian synthetic turbulence (Section~\ref{sect:synthetic_data}) confirms that this qualitative structure transfers robustly beyond the idealized Gaussian setting, providing meaningful gains of 15--30\% over standard RBF kernels even when the data depart from the reference measure.

This framework bridges the gap between rigorous stochastic PDE theory and practical computational tools, demonstrating how modern mathematical results can inform machine learning and data-science methodologies for complex physical systems. More broadly, this work suggests a general paradigm for GP-based modeling of nonlinear dissipative systems: identify linearizations whose invariant measures are equivalent to that of the nonlinear dynamics, and use the resulting linear covariance as a principled prior. While the present application exploits the unique quasi-Gaussianity of 2D stochastic Navier--Stokes, the methodology extends naturally to other systems with known or approximable equilibrium structure, opening new directions for physics-informed probabilistic modeling grounded in rigorous dynamical theory.

\bibliographystyle{plain}    
\bibliography{refs_gp_sns,references,references_hamiltonian,ref}

\end{document}